\RequirePackage[l2tabu, orthodox]{nag} 
\documentclass[11pt, a4paper]{article}

\title{Multistep matrix splitting iteration preconditioning\\for singular linear systems}
\author{Keiichi Morikuni\footnote{\texttt{morikuni@cs.tsukuba.ac.jp}} \thanks{Division of Information Engineering, Faculty of Engineering, Information and Systems, University of Tsukuba}}

\usepackage{amsmath, amsfonts, amssymb, amsthm}
\usepackage[all, warning]{onlyamsmath} 
\usepackage{hyperxmp}
\usepackage[
	bookmarks=true, 
	bookmarksnumbered=true, 
	bookmarksopen=false,
	bookmarkstype=toc,
	colorlinks={true},%
	pdfstartview=/Fit,
	pdfauthor={Keiichi Morikuni},%
	pdfcreator={},%
	pdfdisplaydoctitle={true},%
	pdfkeywords={Preconditioner, Inner-outer iteration, GMRES method, Matrix splitting iterations, Singular linear system},%
	pdftitle={Multistep matrix splitting iteration preconditioning for singular linear systems},
	setpagesize={false},%
	urlcolor={red},
]{hyperref}
\usepackage{graphicx}
\hypersetup{pdfstartview=FitH,pdfpagemode=UseNone}
\usepackage[T1]{fontenc}
\usepackage{subfigure}
\usepackage{exscale, fix-cm, fullpage, lmodern, textcomp}
\usepackage{algorithm, algorithmic}
\usepackage{multirow}

\newcommand{\LINEIFELSE}[3]{%
    \STATE\algorithmicif\ {#1}\ \algorithmicthen\ {#2} \algorithmicelse\ {#3}%
}

\newcommand{\STATELINEFOR}[3]{%
    \STATE {#1}\algorithmicfor\ {#2}\ \algorithmicdo\ {#3} \algorithmicend\ \algorithmicfor%
}

\theoremstyle{plain} 
\newtheorem{theorem}{Theorem}[section]
\newtheorem{lemma}[theorem]{Lemma}

\newtheorem{proposition}[theorem]{Proposition}

\numberwithin{equation}{section}
\makeatletter
	
	\@addtoreset{equation}{section}
\makeatother

\makeatletter
  
  \@addtoreset{figure}{section}
\makeatother

\makeatletter
  
  \@addtoreset{algorithm}{section}
\makeatother

\makeatletter
  
  \@addtoreset{table}{section}
\makeatother

\newfont{\bg}{cmr9 scaled\magstep4}
\newcommand{\bigzerol}{\smash{\lower1.0ex\hbox{\bg 0}}}

\DeclareMathOperator{\diag}{diag}
\DeclareMathOperator{\ind}{index}
\DeclareMathOperator{\rank}{rank}
\DeclareMathOperator{\spn}{span}

\newcommand\I{\operatorname{I}}

\begin{document}
\maketitle

\begin{abstract}
Multistep matrix splitting iterations serve as preconditioning for Krylov subspace methods for solving singular linear systems.
The preconditioner is applied to the generalized minimal residual (GMRES) method and the flexible GMRES (FGMRES) method.
We present theoretical and practical justifications for using this approach.
Numerical experiments show that the multistep generalized shifted splitting (GSS) and Hermitian and skew-Hermitian splitting (HSS) iteration preconditioning are more robust and efficient compared to standard preconditioners for some test problems of large sparse singular linear systems.


\end{abstract}

\section{Introduction} \label{sec:intro}
Consider solving linear systems
\begin{align}
	A \boldsymbol{x} = \boldsymbol{b},
	\label{eq:Ax=b}
\end{align}
where $A \in \mathbb{R}^{n \times n}$ may be singular and $\boldsymbol{b} \in \mathbb{R}^n$.
For solving large sparse linear systems \eqref{eq:Ax=b}, iterative methods are preferred to direct methods in terms of efficiency and memory requirement.
When the problem \eqref{eq:Ax=b} is ill-conditioned, the convergence of iterative methods such as Krylov subspace methods tends to deteriorate and may be accelerated by using preconditioning.
However, well-established preconditioners using incomplete matrix factorizations \cite{Saad1988}, \cite{BaiDuffWathen2001}, \cite{BaiYin2009} require additional memory whose amount is typically comparable to that of the given problem, and may not work in the singular case.

Another approach for preconditioning Krylov subspace methods for solving linear systems is to use a splitting matrix such as the successive overrelaxation (SOR) method \cite{Frankel1950}, \cite{Young1950th}.
Matrix splitting iterations can serve as preconditioning for Krylov subspace methods.

In the singular case, some iterative methods and preconditioners may be infeasible, i.e., they may break down and/or fail to converge.
In this paper, we focus on using GMRES with preconditioning since the method is well-established and fairly well understood in the singular case \cite{Zhang2010}, \cite{HayamiSugihara2011}, \cite{EldenSimoncini2012}.
GMRES applied to the linear system \eqref{eq:Ax=b} with initial iterate $\boldsymbol{x}_0 \in \mathbb{R}^n$ gives the $k$th iterate $\boldsymbol{x}_k$ such that $\| \boldsymbol{b} - A \boldsymbol{x}_k \| = \min_{\boldsymbol{x} \in \boldsymbol{x}_0 + \mathcal{K}_k (A, \boldsymbol{r}_0)} \| \boldsymbol{b} - A \boldsymbol{x} \|$, where $\| \cdot \|$ is the Euclidean norm, $\boldsymbol{r}_0 = \boldsymbol{b} - A \boldsymbol{x}_0$ is the initial residual, and $\mathcal{K}_k (A, \boldsymbol{r}_0) = \spn \lbrace \boldsymbol{r}_0, A \boldsymbol{r}_0, \dots, A^{k-1} \boldsymbol{r}_0 \rbrace$ is the Krylov subspace of order $k$.
Hereafter, denote $\mathcal{K}_k = \mathcal{K}_k (A, \boldsymbol{r}_0)$ for simplicity.

In the singular case, GMRES may fail to determine a solution of \eqref{eq:Ax=b}.
GMRES is said to break down at some step $k$ if $\dim A \mathcal{K}_k < \dim \mathcal{K}_k$ or $\dim \mathcal{K}_k < k$ \cite[p.\ 38]{BrownWalker1997}.
Note that $\dim A \mathcal{K}_k \leq \dim \mathcal{K}_k \leq k$ holds for each $k$.
The dimensions of $A \mathcal{K}_k$ and $\mathcal{K}_k$ are related to the uniqueness of the iterate $\boldsymbol{x}_k$, whereas $\dim \mathcal{K}_k$ is related to the degeneracy of the Krylov subspace method. 
GMRES determines a solution of $A \boldsymbol{x} = \boldsymbol{b}$ without breakdown for all $\boldsymbol{b} \in \mathcal{R}(A)$ and for all $\boldsymbol{x}_0 \in \mathbb{R}^n$ if and only if $A$ is a group (GP) matrix $\mathcal{N}(A) \cap \mathcal{R}(A) = \lbrace \boldsymbol{0} \rbrace$ \cite[Theorem 2.6]{BrownWalker1997}, \cite[Theorem 2.2]{MorikuniHayami2015}, cf.\ \cite[Theorem 4.4.6]{Schneider2005th}, where $\mathcal{N}(A)$ is the null space of $A$ and $\mathcal{R}(A)$ is the range space of $A$.
The condition that $A$ is a GP matrix is equivalent to that the largest size of the Jordan block of $A$ corresponding to eigenvalue $0$ is not larger than one \cite[section 3]{Oldenburger1940}.

Other than Krylov subspace methods, much efforts have been made to study matrix splitting iterations for solving singular linear systems \eqref{eq:Ax=b} (see \cite{Keller1965}, \cite{MeyerPlemmons1977}, \cite{Dax1990}, \cite{BenziSzyld1997}, \cite{Song1999},	 \cite{Yuan2000}, \cite{SongWang2003}, \cite{Cao2004}, \cite{Wang2007a}).
Some of modern matrix splitting iterations were shown to be effectively used as preconditioning for Krylov subspace methods, and can be potentially useful as multistep matrix splitting iteration preconditioning.
For example, see \cite{Bai2010}, \cite{LiLiuPeng2012}, \cite{ChenLiu2013}, \cite{WuLi2014}, \cite{YangWuXu2014} for the Hermitian and skew-Hermitian splitting (HSS) iterations, \cite{WenHuangWang2011} for the triangular and skew symmetric splitting (TSS) iterations, \cite{CaoMiao2016} for the generalized shift splitting (GSS) iterations, and \cite{ZhengBaiYang2009}, \cite{ZhangLuWei2014} for Uzawa methods for singular saddle point problems.
We shed some light on the preconditioning aspect of matrix splitting iterations in the singular case.

Consider applying GMRES to the preconditioned linear system $A P^{-1} \boldsymbol{u} = \boldsymbol{b}$, $\boldsymbol{x} = P^{-1} \boldsymbol{u}$, which is equivalent to $A \boldsymbol{x} = \boldsymbol{b}$, where $P$ is nonsingular and a preconditioning matrix given by multistep matrix splitting iterations.
The right-preconditioned GMRES (RP-GMRES) method with initial iterate $\boldsymbol{x}_0 \in \mathbb{R}^n$ determines the $k$th iterate $\boldsymbol{x}_k$ such that $\| \boldsymbol{b} - A \boldsymbol{x}_k \| = \break \min_{\boldsymbol{x} \in \boldsymbol{x}_0 + \mathcal{K}_k (P^{-1} A, P^{-1} \boldsymbol{r}_0)} \| \boldsymbol{b} - A \boldsymbol{x} \|$, where $\boldsymbol{u}_0 \in \mathbb{R}^n$ and $\boldsymbol{r}_0 = \boldsymbol{b} - A P^{-1} \boldsymbol{u}_0 = \boldsymbol{b} - A \boldsymbol{x}_0$.
On the other hand, the flexible GMRES (FGMRES) method \cite{Saad1993} allows to change the preconditioning matrix for each iteration.
This means that the number of the multistep matrix splitting iterations may vary in GMRES.

The rest of the paper is organized as follows.
In section \ref{sec:iip}, we give sufficient conditions such that GMRES preconditioned by a fixed number of matrix splitting iterations determines a solution without breakdown, a spectral analysis of the preconditioned matrix, and a convergence bound of the method, and discuss the computational complexity of the method.
In section \ref{sec:FGMRES}, we give sufficient conditions such that FGMRES preconditioned by multistep matrix splitting iterations determines a solution without breakdowns in the singular case.
In section \ref{sec:nuex}, we show numerical experiment results on test problems comparing the multistep generalized shift-splitting (GSS) and Hermitian and skew-Hermitian (HSS) matrix splitting iteration preconditioners with the GSS and HSS preconditioners, respectively.
In section \ref{sec:conc}, we conclude the paper.

\section{GMRES preconditioned by a fixed number of matrix splitting iterations} \label{sec:iip}
Consider applying a preconditioner using several steps of matrix splitting iterations to RP-GMRES.
We give its algorithm as follows (cf.\ \cite{DeLongOrtega1995}, \cite{DeLongOrtega1998}). 

\begin{algorithm}
\caption{GMRES method preconditioned by $\ell$ matrix splitting iterations.}
\label{alg:RiGMRES}
\begin{algorithmic}[1]
\STATE Let $\boldsymbol{x}_0 \in \mathbb{R}^n$ be the initial iterate. $\boldsymbol{r}_0 := \boldsymbol{b} - A \boldsymbol{x}_0$, $\beta := \| \boldsymbol{r}_0 \|$, $\boldsymbol{v}_1 := \boldsymbol{r}_0 / \beta$;
\FOR{$k = 1, 2, \dotsc$ until convergence}
\STATE Apply $\ell$ iterations of a matrix splitting to $A \boldsymbol{z} = \boldsymbol{v}_k$ to obtain $\boldsymbol{z}_k = C^{(\ell)} \boldsymbol{v}_k$;
\STATELINEFOR{$\boldsymbol{w} := A \boldsymbol{z}_k$, }{$i = 1, 2, \dots, k$}{$h_{i, k} := (\boldsymbol{v}_i, \boldsymbol{w})$, $\boldsymbol{w} := \boldsymbol{w} - h_{i, k} \boldsymbol{v}_i$}
\LINEIFELSE{$h_{k+1, k} := \| \boldsymbol{w} \| = 0$}{set $m := k$ and go to line 7}{$\boldsymbol{v}_{k+1} := \boldsymbol{w} / h_{k+1, k}$;}
\ENDFOR
\STATE $\boldsymbol{y}_m := \mathrm{arg\,min}_{\boldsymbol{y} \in \mathbb{R}^m} \| \beta \boldsymbol{e}_1 - H_{m+1, m} \boldsymbol{y} \|$, $\boldsymbol{x}_m := \boldsymbol{x}_0 + [\boldsymbol{z}_1, \boldsymbol{z}_2, \dots, \boldsymbol{z}_m] \boldsymbol{y}_m$;
\end{algorithmic}
\end{algorithm}

Here, $C^{(\ell)}$ is the preconditioning matrix given by a fixed number $\ell$ of matrix splitting iterations, $\boldsymbol{e}_i$ is the $i$th column of the identity matrix, and $H_{m+1, m} = \lbrace h_{i, j} \rbrace \in \mathbb{R}^{(m+1) \times m}$.

We next give an expression for the preconditioned matrix $A C^{(\ell)}$ for GMRES with $\ell$ matrix splitting iterations. 
Consider the matrix splitting iterations applied to $A \boldsymbol{z} = \boldsymbol{v}_k$ in line 3.
Note $\boldsymbol{v}_k \in \mathcal{R}(A)$ if $\boldsymbol{b} \in \mathcal{R}(A)$.
Let $M$ be a nonsingular matrix such that $A = M - N$.
Denote the iteration matrix by $H = M^{-1} N$.
Assume that the initial iterate is $\boldsymbol{z}^{(0)} \in \mathcal{N}(H)$, e.g., $\boldsymbol{z}^{(0)} = \boldsymbol{0}$.
Then, the $\ell$th iterate of the matrix splitting iterations is $\boldsymbol{z}^{(\ell)} = H \boldsymbol{z}^{(\ell - 1)} + M^{-1} \boldsymbol{v}_k = \sum_{i = 0}^{\ell - 1} H^i \! M^{-1} \boldsymbol{v}_k$, $\ell \in \mathbb{N}$.
Hence, the multistep matrix splitting iteration preconditioning and preconditioned matrices are $C^{(\ell)} = \sum_{i = 0}^{\ell - 1} H^i \! M^{-1}$ and $A C^{(\ell)} = M^{-1} \sum_{i = 0}^{\ell - 1} H^i (\mathrm{I} - H) M = M^{-1} \! (\mathrm{I} - H^\ell) M = M^{-1} \sum_{i=0}^{\ell-1} (\mathrm{I} - M^{-1} \! A)^i M^{-1} \! A M$, respectively.

We will give sufficient conditions such that GMRES preconditioned by matrix splitting iterations determines a solution of \eqref{eq:Ax=b} without breakdown.
First, we prepare the following

\begin{proposition}[\textrm{\cite{Hensel1926}, \cite[Theorem 1]{Oldenburger1940}, \cite[Theorem 2]{Tanabe1974}}] \label{prop:semiconv}
Let $H$ be a square real matrix.
Then, $H$ is semiconvergent, i.e., $\lim_{i \rightarrow \infty} H^i$ exists, if and only if either $\lambda = 1$ is semisimple, i.e., the algebraic and geometric multiplicities corresponding to $\lambda = 1$ are equal, or $| \lambda | < 1$ holds for all $\lambda \in \sigma(H) = \lbrace \lambda \mid H \boldsymbol{v} = \lambda \boldsymbol{v}, \boldsymbol{v} \in \mathbb{C}^n \backslash \lbrace \boldsymbol{0} \rbrace \rbrace$ the spectrum of $H$.
\end{proposition}

\begin{lemma} \label{th:nonsing_C}
If $H$ is semiconvergent, then $\sum_{i = 0}^{\ell - 1} H^i$ is nonsingular for all $\ell \in \mathbb{N}$.
\end{lemma}
\begin{proof}
Proposition \ref{prop:semiconv} shows that there exists a nonsingular matrix $S$ such that $J = S^{-1} \! H S = \tilde{J} \oplus \mathrm{I}$ is the Jordan canonical form (JCF) of $H$ with $\rho(\tilde{J}) < 1$ for $\tilde{J} \in \mathbb{R}^{r \times r}$, where $\oplus$ denotes the direct sum and $\rho (A) = \max \lbrace |\lambda|: \lambda \in \sigma (A) \rbrace$ is the spectral radius of $A$.
Hence, $\sum_{i = 0}^{\ell - 1} H^i = S \left\lbrace [ (\mathrm{I} - \tilde{J})^{-1} (\mathrm{I} - \tilde{J}^\ell) ] \oplus (\ell \mathrm{I}) \right\rbrace S^{-1}$ holds for all $\ell \in \mathbb{N}$.
Since $1- \lambda^\ell \not = 0$ holds for all $\lambda \in \sigma(\tilde{J})$ and for all $\ell \in \mathbb{N}$, $\mathrm{I} - \tilde{J}^\ell$ is nonsingular and hence $\sum_{i = 0}^{\ell - 1} H^i$ is nonsingular for all $\ell \in \mathbb{N}$.
\qed
\end{proof}

\begin{lemma} \label{th:ind1_C}
If $H$ is semiconvergent, then $\mathrm{I} - H^\ell$ is a GP matrix for all $\ell \in \mathbb{N}$.
\end{lemma}
\begin{proof}
If $\mathrm{O}$ is the zero matrix, then $\mathrm{I} - H^\ell = S [ (\mathrm{I} - \tilde{J}^\ell) \oplus \mathrm{O} ] S^{-1}$.
Since $\mathrm{I} - \tilde{J}^\ell$ is nonsingular, $\mathrm{I} - H^\ell$ is a GP matrix for all $\ell \in \mathbb{N}$.
\qed
\end{proof}

Now we show that GMRES preconditioned by a fixed number of matrix splitting iterations determines a solution of $A \boldsymbol{x} = \boldsymbol{b}$.

\begin{theorem} \label{th:iGMRES}
Assume that the iteration matrix $H$ is semiconvergent.
Then, GMRES preconditioned by multistep matrix splitting iterations $C^{(\ell)}$ defined above determines a solution of $A \boldsymbol{x} = \boldsymbol{b}$ without breakdown for all $\boldsymbol{b} \in \mathcal{R}(A)$, for all $\boldsymbol{x}_0 \in \mathbb{R}^n$, and for all $\ell \in \mathbb{N}$.
\end{theorem}
\begin{proof}
Since $\sum_{i = 0}^{\ell - 1} H^i$ is nonsingular for all $\ell \in \mathbb{N}$ from Lemma \ref{th:nonsing_C}, $C^{(\ell)} = \sum_{i = 0}^{\ell - 1} H^i M^{-1}$ is nonsingular for all $\ell \in \mathbb{N}$.
Hence, the preconditioned linear system $C^{(\ell)} \! A \boldsymbol{x} = C^{(\ell)} \boldsymbol{b}$ is equivalent to $A \boldsymbol{x} = \boldsymbol{b}$.
Since $C^{(\ell)} \! A = \mathrm{I} - H^\ell$ is a GP matrix for all $\ell \in \mathbb{N}$ from Lemma \ref{th:ind1_C}, the theorem follows from \cite[Theorem 2.2]{MorikuniHayami2015}.
\qed
\end{proof}

This theorem gives \cite[Theorem 4.6]{MorikuniHayami2015} as a corollary if the linear system \eqref{eq:Ax=b} is a symmetric and positive semidefinite linear system.

Theorem \ref{th:iGMRES} relies on the property that the preconditioned matrix is GP, which is implied by the semiconvergence of the iteration matrix, irrespective of the property of $A$.
Hence, we may extend the class of singular linear systems that GMRES can solve by combining with preconditioners.
This means that even though $A$ is not a GP matrix, the multistep matrix splitting iteration preconditioned matrix is a GP matrix for $H$ semiconvergent, and GMRES preconditioned by the multistep matrix splitting iterations determines a solution without breakdown (see Theorem \ref{th:GMRESanyind} in Appendix).

Semiconvergence is a simple and convenient property for deciding if a matrix splitting method is feasible as multistep matrix splitting iterations for preconditioning GMRES in the singular case.
Indeed, there are many matrix splitting iterations whose iteration matrix can be semiconvergent.
They are powerful when used as matrix splitting preconditioners for Krylov subspace methods, and potentially useful as multistep matrix splitting iteration preconditioning for GMRES such as the Jacobi, Gauss-Seidel, SOR, and symmetric SOR (SSOR) methods \cite{Dax1990}, extrapolated methods \cite{Song1999}, two-stage methods \cite{Wang2007a}, the GSS method \cite{CaoMiao2016}, and the HSS method and and its variants \cite{Bai2010}, \cite{LiLiuPeng2012}, \cite{ChenLiu2013}, \cite{WuLi2014}, \cite{YangWuXu2014}.

Theorem \ref{th:iGMRES} applies to some trivial examples.
For example, if $A = L + D + L^\mathsf{T}$ is symmetric and positive semidefinite and the SOR splitting matrix is $M = \omega^{-1} (D + \omega L)$, where $D$ is diagonal, $L$ is strictly lower triangular, and $\omega \in \mathbb{R}$, then the SOR iteration matrix $H = M^{-1} N$ is semiconvergent for $\omega \in (0, 2)$ \cite[Theorem 13]{Dax1990}.
On the other hand, if an iteration matrix $H$ is semiconvergent, the extrapolated iteration matrix $(1 - \gamma) \mathrm{I} + \gamma H$ is also semiconvergent for $0 < \gamma < 2 / (1 + \nu (H))$ \cite[Theorem 2.2]{Song1999}.
We will recall conditions such that the GSS and HSS iteration matrices are semiconvergent in sections \ref{sec:GSS} and \ref{sec:HSS}, respectively.
Hence, these multistep matrix splitting iterations can serve as preconditioning for GMRES.

\subsection{Spectral analysis and convergence bound}
Next, consider decomposing GMRES preconditioned by multistep matrix splitting iterations into the $\mathcal{R}(A C^{(\ell)}) = \mathcal{R}(A)$ and $\mathcal{R}(A)^\perp$ components to analyze the spectral property of the preconditioned matrix (cf.\ \cite{HayamiSugihara2011}).
Assume that the iteration matrix $H$ is semiconvergent throughout this subsection and $\boldsymbol{b} \in \mathcal{R}(A)$.
Let $r = \rank \! A$, $Q_1 \in \mathbf{R}^{n \times r}$ such that $\mathcal{R}(Q_1) = \mathcal{R}(A)$, $Q_2 \in \mathbf{R}^{n \times (n-r)}$ such that $\mathcal{R}(Q_2) = \mathcal{R}(A)^\perp$, and $Q = \left[ Q_1, Q_2 \right]$ be orthogonal.
Then, GMRES applied to $A C^{(\ell)} \boldsymbol{u} = \boldsymbol{b}$ can be seen as GMRES applied to $(Q^\mathsf{T} \! A C^{(\ell)} \! Q) Q^\mathsf{T} \! \boldsymbol{u} = Q^\mathsf{T} \! \boldsymbol{b}$, or
\begin{align}
	\begin{bmatrix}
		Q_1^\mathsf{T} \! A C^{(\ell)} Q_1 & Q_1^\mathsf{T} \! A C^{(\ell)} Q_2\\
		\mathrm{O} & \mathrm{O}
	\end{bmatrix}
	\begin{bmatrix}
		Q_1^\mathsf{T} \boldsymbol{u} \\
		Q_2^\mathsf{T} \boldsymbol{u}
	\end{bmatrix}
	\equiv
	\begin{bmatrix} 
		A_{11} & A_{12} \\
		\mathrm{O} & \mathrm{O}
	\end{bmatrix}
	\begin{bmatrix}
		\boldsymbol{u}^1 \\
		\boldsymbol{u}^2
	\end{bmatrix}
	=
	\begin{bmatrix}
		Q_1^\mathsf{T} \boldsymbol{b} \\
		Q_2^\mathsf{T} \! \boldsymbol{b} 
	\end{bmatrix}
	\equiv
	\begin{bmatrix}
		\boldsymbol{b}^1 \\
		\boldsymbol{0}
	\end{bmatrix}.
	\label{eq:decompeq}
\end{align}
Assume that the initial iterate satisfies $\boldsymbol{u}_0 \in \mathcal{R}(A)$.
Then, the $k$th iterate of GMRES applied to \eqref{eq:decompeq} is given by 
\begin{align*}
	Q^\mathsf{T} \boldsymbol{u}_k \equiv
	\begin{bmatrix}
		\boldsymbol{u}_k^1 \\
		\boldsymbol{u}_k^2
	\end{bmatrix} \in Q^\mathsf{T} \boldsymbol{u}_0 + Q^\mathsf{T} \mathcal{K}_k (A C^{(\ell)}, \boldsymbol{r}_0) = 
	\begin{bmatrix}
		\boldsymbol{u}_0^1 \\
		\boldsymbol{0}
	\end{bmatrix} + \mathcal{K}_k \left( \begin{bmatrix}
		A_{1 1} & A_{1 2} \\
		\mathrm{O} & \mathrm{O}
	\end{bmatrix}, 
	\begin{bmatrix}
		\boldsymbol{r}_0^1 \\
		\boldsymbol{0}
	\end{bmatrix} \right)
\end{align*}
which minimizes $\| Q^\mathsf{T} (\boldsymbol{b} - A C^{(\ell)} \boldsymbol{u}_k) \|$, or $\boldsymbol{u}_k^1 \in \boldsymbol{u}_0^1 + \mathcal{K}_k (A_{1 1}, \boldsymbol{r}_0^1)$ which minimizes $\| \boldsymbol{r}_k \| = \| \boldsymbol{b}^1 - A_{1 1} \boldsymbol{u}_k^1 \|$.
This means that $\boldsymbol{u}_k^1$ is equal to the $k$th iterate of GMRES applied to $A_{1 1} \boldsymbol{u}^1 = \boldsymbol{b}^1$ for all $k$ (cf.\ \cite[section 2.5 ]{HayamiSugihara2011}).

Now we give the spectrum of the preconditioned matrix $A C^{(\ell)}$ to present a convergence bound on GMRES preconditioned by multistep matrix splitting iterations.
Let $r = \rank A$.
The $r$ nonzero eigenvalues of $A C^{(\ell)}$ are the eigenvalues of $A_{11}$, since
\begin{align*}
	\det \left(A C^{(\ell)} - \lambda \mathrm{I} \right) & = \det \left( \begin{bmatrix}
						A_{1 1} - \lambda \I_r & A_{1 2} \\
						\mathrm{O} & \lambda \I_{n-r}
					\end{bmatrix}	 \right) 
	= (-\lambda)^{n-r} \det (A_{11} - \lambda \I_r)
\end{align*}
and $A_{11}$ is nonsingular $\Longleftrightarrow$ $A C^{(\ell)}$ is a GP matrix \cite[Theorem 2.3]{HayamiSugihara2011}.
If $\mu$ is an eigenvalue of $H$, then $A C^{(\ell)} = M^{-1} (\mathrm{I} - H^\ell) M$ has an eigenvalue $\lambda = 1 - \mu^\ell$.
From Proposition \ref{prop:semiconv}, $H$ has $r$ eigenvalues such that $| \mu | < 1$ and $n - r$ eigenvalues such that $\mu = 1$.
For $| \mu | < 1$, we obtain $|\lambda - 1| = | \mu |^\ell \leq \nu(H)^\ell < 1$, where $\nu (H) = \max \lbrace | \lambda | : \lambda \in \sigma(H) \backslash \lbrace 1 \rbrace \rbrace$ is the pseudo spectral radius of $H$, i.e., the $r$ eigenvalues of $A C^{(\ell)}$ are in a disk with center at $1$ and radius $\nu(H)^{\ell} < 1$.
For $\mu = 1$, we have $\lambda = 0$, i.e., the remaining $n - r$ eigenvalues are zero.

\begin{theorem} \label{th:bound}
Let $\boldsymbol{r}_k$ be the $k$th residual of GMRES preconditioned by $\ell$ matrix splitting iterations $C^{(\ell)}$ and $T$ be the Jordan basis of $A C^{(\ell)}$.
Assume that $H$ is semi-convergent.
Then, we have $\| \boldsymbol{r}_k \| \leq \kappa(T) \sum_{i=0}^{\tau(k, d)} \binom k i \rho(H)^{k \ell - i} \| \boldsymbol{r}_0 \|$ for all $\boldsymbol{x}_0 \in \mathcal{R}(C^{(\ell)} A)$ and for all $\boldsymbol{b} \in \mathcal{R}(A)$, where $\kappa(T) = \| T \| \| T^{-1} \|$, $d$ is the size of the largest Jordan block corresponding to a nonzero eigenvalue of $C^{(\ell)} A$, and $\tau(k, d) = \min(k, d-1)$.
\end{theorem}
\begin{proof}
Theorem \ref{th:iGMRES} ensures that GMRES preconditioned by multistep matrix splitting iterations determines a solution of $A \boldsymbol{x} = \boldsymbol{b}$ without breakdown for all $\boldsymbol{b} \in \mathcal{R}(A)$ and for all $\boldsymbol{x}_0 \in \mathbf{R}^n$.
From \cite[Theorem 1]{Bai2000}, we have 
\begin{align*}
	\| \boldsymbol{r}_k \| = \min_{p \in \mathbb{P}_k, p(0) = 1} \| p(A C^{(\ell)}) \boldsymbol{r}_0 \| \leq \kappa(T) \min_{p \in \mathbb{P}_k, p(0) = 1} \max_{1 \leq i \leq s} \| p(J_i) \| \| \boldsymbol{r}_0 \|,
\end{align*}
where $\mathbb{P}_k$ is the set of all polynomials of degree not exceeding $k$ and $J_i$ is a Jordan block of $A C^{(\ell)}$ corresponding to a nonzero eigenvalue, $i = 1, 2, \dots, s$.
The second factor is bounded above by $\min_{p \in \mathbb{P}_k, p(0) = 1} \max_{1 \leq i \leq s} \| p(J_i) \| \leq \sum_{i=0}^{\tau(k, d)} \binom k i \rho(H)^{k \ell - i}$ \cite[Theorems 2, 5]{Bai2000}.
\qed
\end{proof}

Note that the residual $\| \boldsymbol{r}_k \|$ does not necessarily depend only on the eigenvalues of $A C^{(\ell)}$ when $\kappa(T)$ is large (see \cite{TebbensMeurant2014} and references therein).

\subsection{Computational complexity}
Compare GMRES for $k \ell$ iterations preconditioned by one step of a matrix splitting method with that for $k$ iterations preconditioned by $\ell$ matrix splitting iterations of the same matrix splitting method in terms of the Krylov subspaces for the iterate $\boldsymbol{x}_{k \ell}$, $\boldsymbol{x}_k$ and computational complexity, since they use the same total number of matrix splitting iterations.

\begin{proposition} \label{prop:spacedim}
If $C^{(\ell)}$ and $H$ are as defined above and $H$ is semiconvergent, then we have 
\begin{align*}
	\mathcal{K}_k (C^{(\ell)} \! A, C^{(\ell)} \boldsymbol{r}_0) \subseteq \mathcal{K}_{k \ell} (C^{(1)} \! A, C^{(1)} \boldsymbol{r}_0).
\end{align*}
\end{proposition}
\begin{proof}
The proof is by induction.
Let $\hat{A} = M^{-1} \! A$ and $\hat{\boldsymbol{r}}_0 = M^{-1} \boldsymbol{r}_0$.
Consider the case $k = 1$.
We have $\mathcal{K}_1 (C^{(\ell)} \! A, C^{(\ell)} \boldsymbol{r}_0) = \spn \lbrace C^{(\ell)} \boldsymbol{r}_0 \rbrace$ and
\begin{align*}
	C^{(\ell)} \boldsymbol{r}_0 = \sum_{i=0}^{\ell-1} (\mathrm{I} - M^{-1} \! A)^i M^{-1} \boldsymbol{r}_0 = \sum_{i=0}^{\ell-1} \sum_{j=0}^{i} \binom{i}{j} (- \hat{A})^j \hat{\boldsymbol{r}}_0.
\end{align*}
Since $C^{(\ell)} \boldsymbol{r}_0 \in \spn \lbrace \hat{\boldsymbol{r}}_0, \hat{A} \hat{\boldsymbol{r}}_0, \dots, \hat{A}^{\ell-1} \hat{\boldsymbol{r}}_0 \rbrace
$, we have
\begin{align*}
	\mathcal{K}_1 (C^{(\ell)} \! A, C^{(\ell)} \boldsymbol{r}_0) \subseteq \mathcal{K}_\ell (C^{(1)} \! A, C^{(1)} \boldsymbol{r}_0).
\end{align*}
Next, assume that $\mathcal{K}_k (C^{(\ell)} \! A, C^{(\ell)} \boldsymbol{r}_0) \subseteq \mathcal{K}_{k \ell} (C^{(1)} \! A, C^{(1)} \boldsymbol{r}_0) = \mathcal{K}_{k\ell} (\hat{A}, \hat{\boldsymbol{r}}_0)$ holds.
Then, 
\begin{align*}
	\mathcal{K}_{k+1} (C^{(\ell)} \! A, C^{(\ell)} \boldsymbol{r}_0) & = \mathcal{K}_k (C^{(\ell)} \! A, C^{(\ell)} \boldsymbol{r}_0) + \spn \lbrace (C^{(\ell)} \! A)^k C^{(\ell)} \boldsymbol{r}_0 \rbrace, \\
	\mathcal{K}_{(k+1)\ell} (C^{(1)} \! A, C^{(1)} \boldsymbol{r}_0) & = \mathcal{K}_{k\ell} (\hat{A}, \hat{\boldsymbol{r}}_0) + 
	\mathcal{K}_k (\hat{A}, \hat{A}^{k \ell} \hat{\boldsymbol{r}}_0).
\end{align*}
From $C^{(\ell)} \! A = \sum_{i=0}^{\ell-1} (\mathrm{I} - \hat{A})^i \hat{A}$, we have 
\begin{align*}
	(C^{(\ell)} \! A)^k C^{(\ell)} \boldsymbol{r}_0 = \left[\sum_{i=0}^{\ell-1} (\mathrm{I} - \hat{A})^i \right]^{k+1} \! \hat{A}^k \hat{\boldsymbol{r}}_0,
\end{align*}
which belongs to $\mathcal{K}_{(k+1)\ell} (\hat{A}, \hat{\boldsymbol{r}}_0) = \mathcal{K}_{(k+1)\ell} (C^{(1)} \! A, C^{(1)} \boldsymbol{r}_0)$.
Hence, \break$\mathcal{K}_{k+1} (C^{(\ell)} \! A, C^{(\ell)} \boldsymbol{r}_0) \subseteq \mathcal{K}_{k \ell} (C^{(1)} \! A, C^{(1)} \boldsymbol{r}_0) = \mathcal{K}_{(k+1) \ell} (\hat{A}, \hat{\boldsymbol{r}}_0)$ holds.
\qed
\end{proof}

This proposition shows that GMRES preconditioned by one matrix splitting iteration gives an optimal Krylov subspace for the iterate, i.e., any Krylov subspace given by GMRES for $k$ iterations preconditioned by $\ell$ matrix splitting iterations is not larger than the one given by GMRES for $k \ell$ iterations preconditioned by one matrix splitting iteration.
However, GMRES preconditioned by one matrix splitting iteration is not necessarily more efficient than GMRES preconditioned by more than one matrix splitting iteration, as will be seen in section \ref{sec:nuex}.
Indeed, while GMRES for $k \ell$ iterations preconditioned by one matrix splitting iteration requires $k \ell$ matrix-vector products of $A$ with $\boldsymbol{z}_k$ and $k \ell$ orthogonalizations, GMRES for $k$ iterations preconditioned by $\ell$ matrix splitting iterations requires $k$ matrix-vector products of $A$ with $\boldsymbol{z}_k$ and $k$ orthogonalizations.
Hence, GMRES for $k \ell$ iterations preconditioned by one matrix splitting iteration needs more computations.
Therefore, GMRES preconditioned by more than one matrix splitting iteration may be more efficient.

Moreover, Proposition \ref{prop:spacedim} gives a lower bound of the number of iterations of GMRES preconditioned by multistep matrix splitting iterations which is required to determine a solution.
Let $s$ be the smallest integer such that $\mathcal{K}_s (C^{(1)} \! A, C^{(1)} \boldsymbol{r}_0) < s$.
Assume that GMRES preconditioned by $\ell$ matrix splitting iterations determines a solution at the $k$th step, where $k$ is the smallest integer such that $\mathcal{K}_k (C^{(\ell)} \! A, C^{(\ell)} \boldsymbol{r}_0) = \mathcal{K}_{s-1} (C^{(1)} \! A, C^{(1)} \boldsymbol{r}_0)$. 
Then, $k$ is larger than $(s-1) / \ell$.
Hence, GMRES preconditioned by $\ell$ matrix splitting iterations requires more than $(s - 1) / \ell$ iterations to determine a solution.

\section{Flexible GMRES preconditioned by multistep matrix splitting iterations.}\label{sec:FGMRES}
The preconditioners given in section \ref{sec:iip} uses a fixed number of matrix splitting iterations.
This can be extended to allow a variable number of matrix splitting iterations for each iteration in line 3, Algorithm \ref{alg:RiGMRES} (flexible GMRES (FGMRES) method \cite{Saad1993}).
Let $C^{(\ell_k)}$ be the multistep matrix splitting iteration preconditioning matrix for the $k$th iteration.
Then, the FGMRES iterate $\boldsymbol{x}_k^\mathrm{F}$ is determined over the space $\boldsymbol{x}_0 + \mathcal{R}(Z_k^\mathrm{F}) = \boldsymbol{x}_0 + \mathcal{R}([C^{(\ell_1)} \boldsymbol{v}_1^\mathrm{F}, C^{(\ell_2)} \boldsymbol{v}_2^\mathrm{F}, \dots, C^{(\ell_k)} \boldsymbol{v}_k^\mathrm{F}])$, which is no longer a Krylov subspace.
Quantities denoted with superscript $\mathrm{F}$ are relevant to FGMRES hereafter.
Hence, Theorem \ref{th:iGMRES} does not apply to FGMRES preconditioned by multistep matrix splitting iterations.

Similarly to the breakdown of GMRES due to the linear dependence of $\boldsymbol{v}_{k+1}$ on $\boldsymbol{v}_1$, $\boldsymbol{v}_2$, \dots, $\boldsymbol{v}_k$, FGMRES may break down with $h_{k+1, k}^\mathrm{F} = 0$ due to the matrix-vector product $A C^{(\ell_k)} \boldsymbol{v}_k^\mathrm{F} = \boldsymbol{0}$, i.e., $\boldsymbol{v}_k^\mathrm{F} \in \mathcal{N}(A C^{(\ell_k)})$, in the singular case.
If $C^{(\ell_k)}$ is nonsingular, then for $\boldsymbol{b} \in \mathcal{R}(A)$, $\boldsymbol{v}_k^\mathrm{F} \not = \boldsymbol{0}$ $\Longleftrightarrow$ $A C^{(\ell_k)} \boldsymbol{v}_k^\mathrm{F} \not = \boldsymbol{0}$ is equivalent to that $A C^{(\ell_k)}$ is a GP matrix, which is given by the iteration matrix $H$ semiconvergent.

Notice that \cite[Proposition 2.2]{Saad1993} holds irrespective of the nonsingularity of $A$: if $\boldsymbol{r}_0 \not = \boldsymbol{0}$, $h_{i+1, i}^\mathrm{F} \not = 0$ for $i = 1, 2, \dots, k-1$, and $H_k^\mathrm{F} = \lbrace h_{i, j}^\mathrm{F} \rbrace \in \mathbb{R}^{k \times k}$ is nonsingular, then $h_{k+1, k}^\mathrm{F} = 0$ is equivalent to that the FGMRES iterate $\boldsymbol{x}_k^\mathrm{F}$ is uniquely determined and is a solution of $A \boldsymbol{x} = \boldsymbol{b}$.
Here, the nonsingularity of $H_k^\mathrm{F}$ is ensured by an additional assumption as follows.
Let ${Q_k}^\mathsf{T} R_{k+1, k} = H_{k+1, k}^\mathrm{F}$ be the QR factorization of $H_{k+1, k}^\mathrm{F}$, where $Q_k$ is the product of Givens rotations $\Omega_k \Omega_{k-1} \cdots \Omega_1$ such as $\Omega_i = \mathrm{I}_{i-1} \oplus \left[
	\begin{smallmatrix}
		c_i & s_i \\
		-s_i & c_i
	\end{smallmatrix}
	\right]
	\oplus \mathrm{I}_{k-i}$ and $R_{k+1, k} \in \mathbb{R}^{(k+1) \times k}$ is upper triangular.
The scalars $c_k$ and $s_k$ are chosen to satisfy $c_k^2 + s_k^2 = 1$ and to vanish the $(k+1, k)$ entry of $\Omega_{k-1} \cdots \Omega_1 H_{k+1, k}^\mathrm{F}$.
It follows from \cite[Lemma 4]{Vuik1995} that if $\| \boldsymbol{v}_{k}^\mathrm{F} - A \boldsymbol{z}_k^\mathrm{F} \| < |c_{k-1}|$ for $c_1 \not = 0$, $c_2 \not = 0$, \dots, $c_{k-1} \not = 0$ and $\boldsymbol{r}_k^\mathrm{F} \not = \boldsymbol{0}$, then $H_k$ is nonsingular.
Thus, we have the following.

\begin{theorem} \label{th:FGMRES}
If the iteration matrix $H$ defined above is semiconvergent and the multistep matrix splitting iterations attain the residual norm $\| \boldsymbol{v}_k^\mathrm{F} - A \boldsymbol{z}_k^\mathrm{F} \| < |c_k|$ for the $k$th iteration, then FGMRES preconditioned by the multistep matrix splitting iterations determines a solution of $A \boldsymbol{x} = \boldsymbol{b}$ for all  $\boldsymbol{b} \in \mathcal{R}(A)$ and for all $\boldsymbol{x}_0 \in \mathbb{R}^n$.
\end{theorem}

\section{Numerical experiments} \label{sec:nuex}
Numerical experiments on the discretized Stokes problem and artificially generated problems show the feasibility of GMRES and FGMRES preconditioned by multistep matrix splitting iterations and the effectiveness of the former.
These methods were compared with previous preconditioners in terms of the central processing unit (CPU) time.
For instance for multistep matrix splitting iteration preconditioning, we used the generalized shift-splitting (GSS) and Hermitian and skew-Hermitian splitting (HSS) and their inexact variants.
Although no condition such that GMRES preconditioned by a fixed number of matrix splitting iterations of an inexact splitting determines a solution without breakdown is given, we used the method for comparisons.

The initial iterates for the multistep matrix splitting iterations and GMRES and FGMRES iterations were set to zero.
No restarts were used for these methods.
The matrix splitting iterations in FGMRES approximately solved the linear system $A \boldsymbol{z} = \boldsymbol{v}_k$ to the accuracy on the residual norm $\| \boldsymbol{v}_k^\mathrm{F} - A \boldsymbol{z}_{k+1}^\mathrm{F} \| < |c_k|$ to ensure that FGMRES determines a solution without breakdown (Theorem \ref{th:FGMRES}).
The stopping criterion used for GMRES and FGMRES iterations was in terms of the relative residual norm $\| \boldsymbol{b} - A \boldsymbol{x}_k \| \leq 10^{-6} \| \boldsymbol{r}_0 \|$.

The computations were done on a computer with an Intel Xeon CPU E5-2670 2.50GHz, 256 GB random-access memory (RAM), and Community Enterprise Operating System (CentOS) Version 6.8.
All programs for the iterative methods were coded and run in Matlab R2014b for double precision floating point arithmetic with unit roundoff $2^{-53} \simeq 1.1 \cdot 10^{-16}$.

\subsection{Multistep generalized shifted splitting iteration preconditioning.} \label{sec:GSS}
We give numerical experiment results on singular saddle point problems 
\begin{align}
	A \boldsymbol{x} = 
	\begin{bmatrix}
		C & B^\mathsf{T} \\
		- B & \mathrm{O}
	\end{bmatrix}
	\boldsymbol{x}
	= \boldsymbol{b}, \quad
	B \in \mathbb{R}^{q \times p}, \quad
	C \in \mathbb{R}^{p \times p}~ \mbox{positive definite},
	\label{eq:saddle}
\end{align}
comparing GMRES preconditioned by $\ell$ iterations and FGMRES preconditioned by $\ell_k$ iterations of the generalized shifted splitting (GSS) 
\begin{align}
	\frac{1}{2}
	\begin{bmatrix}
		\alpha \mathrm{I} + C & B^\mathsf{T} \\
		- B & \beta \mathrm{I}
	\end{bmatrix}
	\boldsymbol{z}^{(i+1)}
	= 
	\frac{1}{2}
	\begin{bmatrix}
		\alpha \mathrm{I} & - B^\mathsf{T} \\
		B & \beta \mathrm{I}
	\end{bmatrix}
	\boldsymbol{z}^{(i)}
	+ \boldsymbol{d},
	\quad i = 1, 2, \dots, \ell ~ \mbox{or} ~ \ell_k
	\label{eq:GSS}
\end{align}
and its inexact variant (IGSS) with GMRES with the standard GSS and IGSS preconditioning $\ell = 1$ and a sparse direct solver, where $\ell$ is the number of GSS and IGSS iterations.

Consider test problems of the form \eqref{eq:saddle} given by the Stokes problem $- \mu \mathrm{\Delta} \boldsymbol{u} + \nabla p = \boldsymbol{f}$, $\nabla \cdot \boldsymbol{u} = 0$ in an open domain $\Omega$ in $\mathbb{R}^2$ with the boundary and normalization conditions  $\boldsymbol{u} = \boldsymbol{0}$ on $\partial \Omega$ and $\int_\Omega p(x) \mathrm{d} x = 0$, respectively, where $\mu$ is the kinematic viscosity constant, $\mathrm{\Delta}$ is the componentwise Laplace operator, the vector field $\boldsymbol{u}$ denotes the velocity, $\nabla$ and $\nabla \cdot$ denote the gradient and divergence operators, respectively, and the scalar function $p$ denotes the pressure.
The Stokes problem was discretized upwind in square domain $\Omega = (0, 1) \times (0, 1)$ on uniform grid.
Thus, the matrix representation of the Stokes problem is $C = (\mathrm{I}_q \otimes T + T \otimes \mathrm{I}_q) \oplus (\mathrm{I}_q \otimes T + T \otimes \mathrm{I}_q) \in \mathbb{R}^{2 q^2 \times 2 q^2}, B^\mathsf{T} = [\hat{B}^\mathsf{T}, \boldsymbol{b}_1, \boldsymbol{b}_2] \in \mathbb{R}^{2 q^2 \times (q^2 + 2)}, \hat{B} = [ (\mathrm{I}_q \otimes F)^\mathsf{T}, (F \otimes \mathrm{I}_q)^\mathsf{T} ] \in \mathbb{R}^{q^2 \times 2 q^2}, T = \mu h^{-2} \mathrm{tridiag} (-1, 2, -1) + (2 h)^{-1} \mathrm{tridiag} (-1, 1, 0) \in \mathbb{R}^{q \times q}, F = h^{-1} \mathrm{tridiag} (-1, 1, 0) \in \mathbb{R}^{q \times q}$, where $\otimes$ denotes the Kronecker product, $\boldsymbol{b}_1^\mathsf{T} = [\boldsymbol{e}^\mathsf{T}, \boldsymbol{0}^\mathsf{T}] \hat{B}$, $\boldsymbol{b}_2^\mathsf{T} = [\boldsymbol{0}^\mathsf{T}, \boldsymbol{e}^\mathsf{T}] \hat{B}$, $\boldsymbol{e} = [1, 1, \dots, 1]^\mathsf{T} \in \mathbb{R}^{q^2 / 2}$, and $h = (q+1)^{-1}$ is the discretization meshsize \cite{BaiGolubPan2004}.
The (2,1) and (1,2) blocks were modified to be rank-deficient as done in \cite[section 5]{ZhengBaiYang2009}, \cite[Example 4.1]{CaoMiao2016}.
We chose two viscosity values $\mu = 10^{-5}$ and $1$ and three kinds of grids $16 \times 16$, $24 \times 24$, and $32 \times 32$.
The right-hand side vector for \eqref{eq:Ax=b} was set to $\boldsymbol{b} = A \boldsymbol{e}$.

The GSS iteration matrix is semiconvergent for $\alpha$, $\beta > 0$ \cite[Theorem 3.2]{CaoMiao2016} and GMRES preconditioned by the multistep GSS iterations determines a solution of $A \boldsymbol{x} = \boldsymbol{b}$ without breakdown for \eqref{eq:saddle}, since $A$ is positive definite (Theorem \ref{th:iGMRES}).
On the other hand, the GMRES methods preconditioned by IGSS and its multistep version are not guaranteed to determine a solution without breakdown.
The value of $\beta$ for GSS and IGSS was set to $\| B \|^2 / \| C \|$ \cite{CaoLiYao2015}.
The value of $\alpha$ for GSS and IGSS was experimentally determined to have the minimal CPU time.
The resulting values were $\alpha = 10$, $13$, and $15$ for $\mu = 1$ with grids $16$, $24$, and $32$, respectively, and $\alpha = 30$, $37$, and $57$ for $\mu = 10^{-5}$ with grids $16$, $24$, and $32$, respectively.

The linear system \eqref{eq:GSS} was solved via \cite[Algorithm 2.1]{CaoLiYao2015} by using the LU factorization for GSS and was solved by using GMRES with the stopping criterion $10^{-1}$ in terms of the relative residual norm for the inexact multistep GSS (IGSS) iteration preconditioning \cite[Section 6]{BaiGolubNg2008}.

Tables \ref{table:Stokes1} and \ref{table:Stokes1e-5} give the number of iterations and the CPU time in seconds for the Stokes problem with $\mu = 1$ and $10^{-5}$, respectively.
Iter denotes the number of GMRES iterations and Time denotes the CPU time in seconds.
GMRES, GSS, IGSS, F-GSS, F-IGSS, and mldivide denote GMRES with no preconditioning, GMRES preconditioned by the multistep GSS iterations, its inexact variant, FGMRES preconditioned by the multistep GSS iterations, its inexact variant, and the Matlab direct solver function mldivide, respectively.
Hence, the CPU time for GMRES preconditioned by the multistep GSS iterations will improve with a sophisticated choice of the value of $\ell$.
The least CPU time for each number of grids among the iterative methods is denoted by bold texts.

The number of GSS and IGSS iterations was set to three throughout for simplicity, which is not necessarily optimal in terms of the CPU time.
For example, GMRES preconditioned by six GSS iterations took 76.67 seconds to attain the stopping criterion for the problem with $\mu = 10^{-5}$ and $q = 36$.

Table \ref{table:Stokes1} shows that for well-conditioned problems $\mu = 1$, IGSS ($\ell = 1$) took the least CPU time to attain the stopping criterion among the iterative methods except for the small problem with grids $16 \times 16$.
Table \ref{table:Stokes1e-5} shows that for ill-conditioned problems $\mu = 10^{-5}$, GSS ($\ell = 3$) took the least CPU time among the iterative methods.
For small problems with grids $16 \times 16$ and $24 \times 24$, FGMRES took larger CPU time than other iterative methods.
For ill-conditioned problems $\mu = 10^{-5}$, although FGMRES required the fewest numbers of iterations, it did not outperform other methods in terms of the CPU time.
The cost for solving the linear system $A \boldsymbol{z} = \boldsymbol{v}_k$ with the stopping criterion $\| \boldsymbol{v}_k^\mathrm{F} - A \boldsymbol{z}_{k+1}^\mathrm{F} \| < | c_k |$ (Theorem \ref{th:FGMRES}) was not marginal in FGMRES, since the value of $| c_k |$ becomes small in the final FGMRES iterations.
Note that the Matlab direct solver mldivide function gave more accurate solutions than the iterative methods within less CPU time.

\begin{table}[]
	\caption{Number of iterations and CPU time for \eqref{eq:saddle} with $\mu = 1$.}
	\centering	
	\footnotesize
	\begin{tabular}{l|*{2}{r}|*{2}{r}|*{2}{r}}
		\multicolumn{1}{l|}{Grids} & \multicolumn{2}{c|}{$16 \times 16$} & \multicolumn{2}{c|}{$24 \times 24$} & \multicolumn{2}{c}{$32 \times 32$} \\
		\hline
		 & \multicolumn{1}{l}{Iter} & \multicolumn{1}{l|}{Time} & \multicolumn{1}{l}{Iter} & \multicolumn{1}{l|}{Time} & \multicolumn{1}{l}{Iter} & \multicolumn{1}{l}{Time} \\
		\hline
		GMRES & 145 & 0.279 & 212 & 0.321 & 310 & 2.505 \\
		GSS ($\ell = 1$) & 19 & 0.071 & 20 & 0.315 & 23 & 3.311 \\
		GSS ($\ell = 3$) & 13 & \textbf{0.057} & 15 & 0.336 & 17 & 3.311 \\
		IGSS ($\ell = 1$) & 18 & 0.136 & 19 & \textbf{0.258} & 21 & \textbf{1.088} \\
		IGGS ($\ell = 3$) & 15 & 0.170 & 17 & 0.699 & 19 & 2.966 \\
		F-GSS & 29 & 0.062 & 37 & 0.362 & 40 & 3.238 \\
		F-IGSS & 29 & 0.395 & 38 & 0.739 & 39 & 4.132 \\
		\hline
		mldivide & & 0.011 & & 0.034 & & 0.066
	\end{tabular}
	\label{table:Stokes1}
	\vspace{11pt}	
	
	\caption{Number of iterations and CPU time for \eqref{eq:saddle} with $\mu = 10^{-5}$.}
	\centering	
	\footnotesize
	\begin{tabular}{l|*{2}{r}|*{2}{r}|*{2}{r}}
		\multicolumn{1}{l|}{Grids} & \multicolumn{2}{c|}{$16 \times 16$} & \multicolumn{2}{c|}{$24 \times 24$} & \multicolumn{2}{c}{$32 \times 32$} \\
		\hline
		 & \multicolumn{1}{l}{Iter} & \multicolumn{1}{l|}{Time} & \multicolumn{1}{l}{Iter} & \multicolumn{1}{l|}{Time} & \multicolumn{1}{l}{Iter} & \multicolumn{1}{l}{Time} \\
		\hline
		GMRES & 766 & 2.915 & 1,723 & 17.51 & 3,861 & 391.3 \\
		GSS ($\ell = 1$) & 740 & 2.693 & 1,585 & 16.32 & 3,036 & 275.2 \\
		GSS ($\ell = 3$) & 561 & \textbf{1.885} & 1,130 & \textbf{10.64} & 1,549 & \textbf{107.2} \\
		IGSS ($\ell = 1$) & 748 & 5.668 & 1,587 & 42.62 & 3,026 & 406.6 \\
		IGSS ($\ell = 3$) & 594 & 7.989 & 1,155 & 59.03 & 1,586 & 276.9 \\	
		F-GSS & 32 & 17.17 & 35 & 60.29 & 35 & 451.0 \\
		F-IGSS & 33 & 295.5 & 37 & 500.8 & 37 & 4113. \\
		\hline
		mldivide & & 0.011 & & 0.015 & & 0.029
	\end{tabular}
	\label{table:Stokes1e-5}
\end{table}

\subsection{Multistep Hermitian and skew-Hermitian splitting iteration preconditioning.} \label{sec:HSS}
We give numerical experiment results on singular and positive semidefinite linear system for \eqref{eq:Ax=b} with the generalized saddle point structure 
\begin{align}
	A \boldsymbol{x} = 
	\begin{bmatrix}
		C & B \\
		-B^\mathsf{T} & G
	\end{bmatrix} \boldsymbol{x} = \boldsymbol{b},
	\label{eq:generalizedsaddlepoint}
\end{align}
where $C \in \mathbb{R}^{p \times p}$ and $G \in \mathbb{R}^{q \times q}$ are symmetric and positive semidefinite, and $B \in \mathbb{R}^{p \times q}$.
The results compare GMRES preconditioned by $\ell$ iterations and FGMRES preconditioned by $\ell_k$ iterations of the Hermitian and skew-Hermitian splitting (HSS) \cite{BaiGolubNg2003}
\begin{align}
	\begin{cases}
		(\alpha \mathrm{I} + \mathcal{H}) \boldsymbol{z}^{(i+1/2)} = (\alpha \mathrm{I} - \mathcal{S}) \boldsymbol{z}^{(i)} + \boldsymbol{v}_k, \\
		(\alpha \mathrm{I} + \mathcal{S}) \boldsymbol{z}^{(i+1)} = (\alpha \mathrm{I} - \mathcal{H}) \boldsymbol{z}^{(i+1/2)} + \boldsymbol{v}_k,
	\end{cases}
	\quad i = 1, 2, \dots, \ell ~ \mbox{or} ~ \ell_k \label{eq:HSS}
\end{align}
and its inexact variant (IHSS) with GMRES with no preconditioning and the standard HSS and IHSS preconditioning $\ell = 1$ for \eqref{eq:generalizedsaddlepoint} and a sparse direct solver, where $\mathcal{H} = (A + A^\mathsf{T}) / 2$, $\mathcal{S} = (A - A^\mathsf{T}) / 2$, and $\alpha \in \mathbb{R}$.
The former and latter systems of \eqref{eq:HSS} were solved by using the Cholesky and LU factorizations, respectively, for HSS, and by using the conjugate gradient (CG) method \cite{HestenesStiefel1952} and the LSQR method \cite{PaigeSaunders1982a}, respectively, with the maximum number of iterations $n$, with the initial iterate equal to zero, and with the stopping criterion $10^{-1}$ in terms of the relative residual 2-norm for the IHSS iterations \cite{BaiGolubNg2008}.
The maximum number of IHSS iterations for FGMRES was $n$.

We generated test problems with the structure
\begin{align}
	(U \oplus V)^\mathsf{T} \! A (U \oplus V) = 
	\begin{bmatrix}
		\hat{C} \oplus \mathrm{O} & \hat{B} \oplus \mathrm{O} \\
		- \hat{B}^\mathsf{T} \oplus \mathrm{O} & \hat{G} \oplus \mathrm{O}
	\end{bmatrix}, \quad
	\boldsymbol{b} = A [1, 2, \dots, n]^\mathsf{T},
	\label{eq:structure}
\end{align}
where $U \in \mathbb{R}^{p \times p}$ and $V \in \mathbb{R}^{q \times q}$ are orthogonal matrices such that $U^\mathsf{T} \! C U = \hat{C} \oplus \mathrm{O}$, $V^\mathsf{T} \! G V = \hat{G} \oplus \mathrm{O}$, and $U^\mathsf{T} B V = \hat{B} \oplus \mathrm{O}$.
We set $U^\mathsf{T} \! C U = \diag (\varphi(1), \varphi(2), \dots, \varphi(p-q-1)) \oplus \mathrm{O} \in \mathbb{R}^{p \times p}$ for $\varphi(i) = \kappa^{i / (p - q - 1)}$, $j \in \mathbb{N}$, $V^\mathsf{T} \! G V = \diag (\psi(1), \psi(2), \dots, \psi(q-2)) \oplus \mathrm{O} \in \mathbb{R}^{q \times q}$ for $\psi(i) = \kappa^{i / (q - 1)}$, and $V^\mathsf{T} \! B^\mathsf{T} U = [ V^\mathsf{T} \! G^\mathsf{T} V, \mathrm{O} ] \in \mathbb{R}^{p \times q}$, where $\kappa = 10^{-j}$.
We set $q = 16$, $32$, and $64$, $p = q^2$, nonzero density $0.1$\% of $A$, and $j = 3$, $6$, and $9$ to show the effect of the condition number on the convergence.
The value of $j$ determines the condition number of $A$ such as $\| A \| \| A^\dag \| = \sqrt{2} \times 10^j$, where $A^\dag$ is the pseudo inverse of $A$.
The orthogonal matrices $U \in \mathbb{R}^{p \times p}$ and $V \in \mathbb{R}^{q \times q}$ were the products of random Givens rotations.
Hence, the HSS iteration matrix $(\alpha \mathrm{I} + \mathcal{S})^{-1} (\alpha \mathrm{I} - \mathcal{H}) (\alpha \mathrm{I} + \mathcal{H})^{-1} (\alpha \mathrm{I} - \mathcal{S})$ is semiconvergent \cite[Theorem 3.6]{Bai2010}, and GMRES preconditioned by multistep HSS iterations determines a solution of \eqref{eq:generalizedsaddlepoint} without breakdown (Theorem \ref{th:iGMRES}).
On the other hand, the GMRES methods preconditioned by multistep IHSS and its iterations are not guaranteed to determine a solution without breakdown.

Multistep HSS and IHSS iterations involve two parameters: the iteration parameter $\alpha$ and the number of HSS iterations $\ell$.
Several techniques were proposed for estimating an optimal value of the HSS iteration parameter for the nonsingular case.
As pointed out by a referee, parameter estimation techniques proposed in \cite{BaiGolubLi2006}, \cite{Bai2009}, \cite{Huang2015}, \cite{Chen2015} are developed for the present case.
Huang's technique need not modify for the singular case.
In Chen's technique, the minimum eigenvalue of $\mathcal{H}$ and the minimum singular value of $\mathcal{S}$ were replaced by the nonzero ones.
After the value of the iteration parameter $\alpha$ was determined, the number of iterations $\ell$ was determined by applying the HSS iterations alone to \eqref{eq:generalizedsaddlepoint} and adopted the smallest between 10 and the smallest number of $i$ which satisfies the relative difference norm $\| \boldsymbol{z}^{(i-1)} - \boldsymbol{z}^{(i)} \| < 10^{-1} \| \boldsymbol{z}^{(i)} \|$.
The CPU times required by Huang's technique to determine the values of the HSS iteration parameter were 0.001 seconds for $q = 16$, 0.002 seconds for $q = 32$, and 0.019 seconds for $q = 64$.
Bai et al.'s technique \cite{BaiGolubLi2006} and Bai's technique \cite{Bai2009} did not give more reasonable values of the HSS iteration parameter than Huang's \cite{Huang2015} and Chen's \cite{Chen2015} techniques for the test problems.
Note that Bai's technique \cite{Bai2009} is for the saddle-point problem instead of the generalized saddle-point problem \eqref{eq:generalizedsaddlepoint}, and does not take into account the $(2, 2)$ block of \eqref{eq:generalizedsaddlepoint} for the estimation.

Tables \ref{tbl:sprad_optpar3}--\ref{tbl:sprad_optpar9} give the optimal and estimated values of the HSS iteration parameter and the value of the corresponding pseudo spectral radius of the HSS iteration matrix.
The optimal value of the HSS iteration parameter $\alpha_\mathrm{exp}$ was experimentally determined to minimized the pseudo-spectral radius of the HSS iteration matrix.
The values of the HSS iteration parameters which were estimated by using Huang's and Chen's techniques are denoted with subscript $\mathrm{C}$ and $\mathrm{H}$, respectively.
Chen's technique estimated the values of the parameter close to the optimal one of the parameter which were experimentally determined.

Tables \ref{table:3}--\ref{table:9} give the number of the iterations and the CPU time in seconds for the test problems with different sizes and condition numbers.
HSS, IHSS, HSS$^\prime$, IHSS$^\prime$, F-HSS$^\prime$, and F-IHSS$^\prime$ denote GMRES preconditioned by the HSS preconditioner, its inexact variant, GMRES preconditioned by the multistep HSS iterations, its inexact variants, FGMRES preconditioned by multistep HSS iterations, and its inexact variant, respectively.
$\dag$ means that CG or LSQR for the linear systems did not attain the stopping criterion within $n$ iterations or the IHSS iterations did not satisfy $\| \boldsymbol{v}_k^\mathrm{F} - A \boldsymbol{z}_k^\mathrm{F} \| < | c_k |$ within $n$ iterations for the indicated number of iterations.
$\ddag$ means that the Matlab direct solver mldivide function fails to give a solution, i.e., some of its entries are Not a Number (NaN).

HSS$^\prime$ took the least CPU time to attain the stopping criterion among the iterative methods except for the case $(j, q) = (9, 16)$.
Bai and Chen's techniques tended to give reasonable values of the HSS iteration parameter for well-conditioned or small problems such as the cases $(j, q) = (3, 16)$, $(3, 64)$, $(6, 16)$, $(6, 32)$, whereas Huang's technique tended to give reasonable values of the HSS iteration parameter for ill-conditioned or large problems such as the cases $(j, q) = (6, 64)$, 
\begin{table}[t]
	\caption{HSS parameter $\alpha$ and pseudo spectral radius $\nu (H(\alpha))$ of the iteration matrix of \eqref{eq:HSS} for \eqref{eq:generalizedsaddlepoint}, $j = 3$.}
	\centering	
	\footnotesize
	\begin{tabular}{c|*{3}{r}}
		$q$ & \multicolumn{1}{c}{$16$} & \multicolumn{1}{c}{$32$} & \multicolumn{1}{c}{$64$} \\
		\hline
		$\alpha_\mathrm{exp}$ & \texttt{0.03162} & \texttt{0.03162} & \texttt{0.03162} \\
		$\nu (H(\alpha_\mathrm{exp}))$ & \texttt{0.93869} & \texttt{0.93869} & \texttt{0.93869} \\
		$\alpha_\mathrm{H}$ & \texttt{0.15678} & \texttt{0.08055} & \texttt{0.03295} \\
		$\nu (H(\alpha_\mathrm{H}))$ & \texttt{0.98732} & \texttt{0.97548} & \texttt{0.94109} \\
		$\alpha_\mathrm{C}$ & \texttt{0.03162} & \texttt{0.03162} & \texttt{0.03162} \\
		$\nu (H(\alpha_\mathrm{C}))$ & \texttt{0.93869} & \texttt{0.93869} & \texttt{0.93869} 
	\end{tabular}
	\label{tbl:sprad_optpar3}
	\vspace{8pt}
	\caption{HSS parameter $\alpha$ and pseudo spectral radius $\nu (H(\alpha))$ of the iteration matrix of\eqref{eq:HSS} for \eqref{eq:generalizedsaddlepoint}, $j = 6$.}
	\centering	
	\footnotesize
	\begin{tabular}{c|*{3}{r}}
		$q$ & \multicolumn{1}{c}{$16$} & \multicolumn{1}{c}{$32$} & \multicolumn{1}{c}{$64$} \\
		\hline
		$\alpha_\mathrm{exp}$ & \texttt{0.00100} & \texttt{0.00100} & \texttt{0.00100} \\
		$\nu (H(\alpha_\mathrm{exp}))$ & \texttt{0.99800} & \texttt{0.99800} & \texttt{0.99800} \\
		$\alpha_\mathrm{H}$ & \texttt{0.14289} & \texttt{0.08068} & \texttt{0.03610} \\
		$\nu (H(\alpha_\mathrm{H}))$ & \texttt{0.99999} & \texttt{0.99998} & \texttt{0.99994} \\
		$\alpha_\mathrm{C}$ & \texttt{0.00100} & \texttt{0.00100} & \texttt{0.00100} \\
		$\nu (H(\alpha_\mathrm{C}))$ & \texttt{0.99800} & \texttt{0.99800} & \texttt{0.99800} \\
	\end{tabular}
	\label{tbl:sprad_optpar6}
	\vspace{8pt}
	\caption{HSS parameter $\alpha$ and pseudo spectral radius $\nu (H(\alpha))$ of \eqref{eq:HSS} the iteration matrix for \eqref{eq:generalizedsaddlepoint}, $j = 9$.}
	\centering	
	\footnotesize
	\begin{tabular}{c|*{3}{r}}
		$q$ & \multicolumn{1}{c}{$16$} & \multicolumn{1}{c}{$32$} & \multicolumn{1}{c}{$64$} \\
		\hline
		$\alpha_\mathrm{exp}$ & \texttt{3.16e-5} & \texttt{3.16e-5} & \texttt{3.16e-5} \\
		$\nu (H(\alpha_\mathrm{exp}))$ & \texttt{0.99993} & \texttt{0.99993} & \texttt{0.99993} \\
		$\alpha_\mathrm{H}$ & \texttt{0.14102} & \texttt{0.13766} & \texttt{0.03878} \\
		$\nu (H(\alpha_\mathrm{H}))$ & \texttt{1.00000} & \texttt{1.00000} & \texttt{1.00000} \\
		$\alpha_\mathrm{C}$ & \texttt{3.16e-5} & \texttt{3.16e-5} & \texttt{3.16e-5} \\
		$\nu (H(\alpha_\mathrm{C}))$ & \texttt{0.99993} & \texttt{0.99993} & \texttt{0.99993} \\
	\end{tabular}
	\label{tbl:sprad_optpar9}
\end{table}
\begin{table}[]
	\caption{Number of iterations and CPU time for \eqref{eq:generalizedsaddlepoint}, \eqref{eq:structure} with $j = 3$.}
	\centering	
	\footnotesize
	\begin{tabular}{cl|*{3}{r}|*{3}{r}|*{3}{r}}
		& \multicolumn{1}{c|}{$q$} & \multicolumn{3}{c|}{$16$} & \multicolumn{3}{c|}{$32$} & \multicolumn{3}{c}{$64$} \\
		\cline{2-11}	
		& & \multicolumn{1}{c}{$\ell$} & \multicolumn{1}{l}{Iter} & \multicolumn{1}{l|}{Time} & \multicolumn{1}{c}{$\ell$} & \multicolumn{1}{l}{Iter} & \multicolumn{1}{l|}{Time} & \multicolumn{1}{c}{$\ell$} & \multicolumn{1}{l}{Iter} & \multicolumn{1}{l}{Time}  \\
		\hline
        & GMRES   &   & 112 &          0.061 &    & 159 &          0.154 &    & 167 &          0.919 \\
		\hline
		\multirow{6}{17pt}{$\alpha_\mathrm{exp}$\\$\alpha_\mathrm{C}$} & HSS     & &   47 &          0.018 &  & 68 &          0.049 & & 81 &          0.493 \\
		& HSS$^\prime$     & 7 &  19 & \textbf{0.010} & 10 &  15 &          0.039 & 10 &  15 & \textbf{0.405} \\
		& IHSS             &   &  52 &          0.089 &    &  56 &          0.137 &    &  70 &          0.834 \\
		& IHSS$^\prime$    & 7 &  29 &         $\dag$ & 10 &   1 &         $\dag$ & 10 &   1 &         $\dag$ \\
		& F-HSS$^\prime$   &   &  19 &          0.014 &    &   1 &         $\dag$ &    &   1 &         $\dag$ \\
		& F-IHSS$^\prime$  &   &   8 &          0.267 &    &  11 &          0.349 &    &  11 &          1.442 \\
		\hline
		\multirow{6}{17pt}{$\alpha_\mathrm{H}$} & HSS & & 63 &  0.026 & & 62 &  0.043 & & 79 &  0.453 \\
		& HSS$^\prime$    & 5 &   22 &         0.011 & 8 &  15 &  \textbf{0.033} & 10 &  15 & \textbf{0.405} \\
		& IHSS            &   &   65 &         0.083 &   &  63 &           0.129 &    &  64 &          0.808 \\
		& IHSS$^\prime$   & 5 &    1 &        $\dag$ & 8 &   1 &          $\dag$ & 10 &   1 &         $\dag$ \\
		& F-HSS$^\prime$  &   &   35 &         0.017 &   &   1 &          $\dag$ &    &   1 &         $\dag$ \\
		& F-IHSS$^\prime$ &   &   14 &         0.337 &   &  15 &           0.425 &    &  13 &          1.327 \\
		\hline
		& mldivide & & & 0.000 & & & 0.001 & & & \ddag 0.011		
	\end{tabular}
	\label{table:3}
	\vspace{11pt}	
	
	\caption{Number of iterations and CPU time for \eqref{eq:generalizedsaddlepoint}, \eqref{eq:structure} with $j = 6$.}
	\centering	
	\footnotesize
	\begin{tabular}{cl|*{3}{r}|*{3}{r}|*{3}{r}}
		& \multicolumn{1}{c|}{$q$} & \multicolumn{3}{c|}{$16$} & \multicolumn{3}{c|}{$32$} & \multicolumn{3}{c}{$64$} \\
		\cline{2-11}		
		& & \multicolumn{1}{c}{$\ell$} & \multicolumn{1}{l}{Iter} & \multicolumn{1}{l|}{Time} & \multicolumn{1}{c}{$\ell$} & \multicolumn{1}{l}{Iter} & \multicolumn{1}{l|}{Time} & \multicolumn{1}{c}{$\ell$} & \multicolumn{1}{l}{Iter} & \multicolumn{1}{l}{Time}  \\
		\hline
		& GMRES  &    & 163 &          0.117 &    & 464 &          1.136 &    & 1,014 &          31.50 \\
		\hline
		\multirow{6}{17pt}{$\alpha_\mathrm{exp}$\\$\alpha_\mathrm{C}$} & HSS & & 126 & 0.073 & & 209 & 0.262 & & 373 & 4.634 \\
		& HSS$^\prime$   & 10 &  49 & \textbf{0.046} & 10 &  71 & \textbf{0.164} & 10 &  109 &          2.139 \\
		& IHSS           &    & 130 &          0.517 &    & 262 &          1.575 &    &  738 &          36.84 \\
		& IHSS$^\prime$  & 10 &   1 &         $\dag$ & 10 &   1 &         $\dag$ & 10 &    1 &         $\dag$ \\
		& F-HSS$^\prime$ &    &   1 &         $\dag$ &    &   1 &         $\dag$ &    &    1 &         $\dag$ \\
		& F-IHSS$^\prime$ &    &   4 &         $\dag$ &    &   2 &         $\dag$ &    &    1 &         $\dag$ \\
		\hline
		\multirow{6}{17pt}{$\alpha_\mathrm{H}$} & HSS & & 162 & 0.116 & & 436 & 0.999 & & 326 &  3.397 \\
		& HSS$^\prime$    &  5 & 146 &          0.136 &  7 & 150 &          0.309 & 10 &  76 & \textbf{1.520} \\
		& IHSS            &    & 163 &          0.214 &    & 444 &          1.444 &    & 362 &          6.001 \\
		& IHSS$^\prime$   &  5 &   1 &         $\dag$ &  7 &   1 &         $\dag$ & 10 &   1 &         $\dag$ \\
		& F-HSS$^\prime$  &    &  32 &         $\dag$ &    &   1 &         $\dag$ &    &   1 &         $\dag$ \\
		& F-IHSS$^\prime$ &    &   6 &         $\dag$ &    &  11 &         $\dag$ &    &  12 &         $\dag$ \\
		\hline
		& mldivide & & & 0.000 & & & \ddag 0.001 & & & \ddag 0.011		
	\end{tabular}
	\label{table:6}
	\vspace{11pt}
	
	\caption{Number of iterations and CPU time for \eqref{eq:generalizedsaddlepoint}, \eqref{eq:structure} with $j = 9$.}
	\centering	
	\footnotesize
	\begin{tabular}{cl|*{3}{r}|*{3}{r}|*{3}{r}}
		& \multicolumn{1}{c|}{$q$} & \multicolumn{3}{c|}{$16$} & \multicolumn{3}{c|}{$32$} & \multicolumn{3}{c}{$64$} \\
		\cline{2-11}
		& & \multicolumn{1}{c}{$\ell$} & \multicolumn{1}{l}{Iter} & \multicolumn{1}{l|}{Time} & \multicolumn{1}{c}{$\ell$} & \multicolumn{1}{l}{Iter} & \multicolumn{1}{l|}{Time} & \multicolumn{1}{c}{$\ell$} & \multicolumn{1}{l}{Iter} & \multicolumn{1}{l}{Time} \\
		\hline
		& GMRES   &    &  117 &          0.062 & & 353 & 0.672 & & 852 &          23.80 \\
		\hline
		\multirow{6}{17pt}{$\alpha_\mathrm{exp}$\\$\alpha_\mathrm{C}$} & HSS & & 169 & 0.125 & & 511 & 1.333 & & 1,308 & 49.83 \\ 
		& HSS$^\prime$    & 10 &  162 &          0.219 & 10 & 317 &          1.063 & 10 & 487 & 13.61 \\
		& IHSS            &    &  107 &          0.717 &    & 297 &          3.455 &    & 938 &          76.97 \\  
		& IHSS$^\prime$   & 10 &    1 &         $\dag$ & 10 &   1 &         $\dag$ & 10 &   1 &         $\dag$ \\  
		& F-HSS$^\prime$  &    &    1 &         $\dag$ &    &   1 &         $\dag$ &    &   1 &         $\dag$ \\   
		& F-IHSS$^\prime$ &    &    1 &         $\dag$ &    &   1 &         $\dag$ &    &   1 &         $\dag$ \\
		\hline
		\multirow{6}{17pt}{$\alpha_\mathrm{H}$} & HSS & & 117 & \textbf{0.064} & & 348 & 0.659 & & 346 & 4.169 \\ 
		& HSS$^\prime$    & 5 & 111 &  0.089 & 6 & 190 &  \textbf{0.389} & 10 &  80 & \textbf{1.503} \\
		& IHSS            &   & 117 &  0.135 &   & 350 &           0.895 &    & 381 &          6.628 \\
		& IHSS$^\prime$   & 5 &   1 & $\dag$ & 6 &   1 &          $\dag$ & 10 &   1 &         $\dag$ \\
		& F-HSS$^\prime$  &   &  30 & $\dag$ &   &   1 &          $\dag$ &    &   1 &         $\dag$ \\
		& F-IHSS$^\prime$ &   &   9 & $\dag$ &   &   8 &          $\dag$ &    &  12 &         $\dag$ \\
		\hline
		& mldivide & & & 0.000 & & & 0.001 & & & \ddag 0.011
	\end{tabular}
	\label{table:9}
\end{table}
$(9, 16)$, $(9, 32)$, $(9, 64)$.
Although F-IHSS$^\prime$ took the fewest numbers of iterations, it did not outperform other methods in terms of the CPU time.
IHSS$^\prime$ did not converge for all test problems.
The Matlab direct solver mldivide function failed to give a solution for the large cases $q = 64$, although it outperformed the iterative methods for the other cases, except for the case $j = 6$.

Comparing Tables \ref{table:3}--\ref{table:9} with Tables \ref{tbl:sprad_optpar3}--\ref{tbl:sprad_optpar9}, we see that these estimated optimal values of the HSS iteration parameter in terms of the pseudo spectral radius did not give optimal CPU time for HSS$^\prime$.
This implies that a small pseudo-spectral radius does not necessarily gives a fast convergence of HSS, HSS$^\prime$, and IHSS (see also Theorem \ref{th:bound}).

\section{Conclusions} \label{sec:conc}
We considered applying several steps of matrix splitting iterations as a preconditioner to GMRES and FGMRES for solving singular linear systems.
We gave sufficient conditions such that GMRES and FGMRES preconditioned by multistep matrix splitting iterations determine a solution without breakdown, and a convergence bound of GMRES preconditioned by multistep matrix splitting iterations based on a spectral analysis.
We presented a complexity issue of using multistep matrix splitting iteration preconditioning more than one step for GMRES.
Numerical experiments showed that GMRES preconditioned by the multistep GSS and HSS iterations is efficient compared to previous methods including FGMRES for large and ill-conditioned problems.

\section*{Appendix}

If $\ind (A) = \min \lbrace d \in \mathbb{N}_0| \rank A^d = \rank A^{d+1} \rbrace$, where $A^0= \mathrm{I}$ and $\mathrm{I}$ is the identity matrix \cite[Definition 7.2.1]{CampbellMeyer2009}, then $d \geq \ind (A)$ is equivalent to $\mathcal{R}(A^d) \cap \mathcal{N}(A^d) = \lbrace \boldsymbol{0} \rbrace$ \cite[Lemma 7.6.1]{CampbellMeyer2009}.
The following theorem gives conditions such that GMRES determines a solution without breakdown for \break $\ind (A) \geq 1$.

\begin{theorem} \label{th:GMRESanyind}
GMRES determines a solution of $A \boldsymbol{x} = \boldsymbol{b}$ without breakdown for all $\boldsymbol{b} \in \mathcal{R}(A^d)$ and for all $ \boldsymbol{x}_0 \in \mathcal{R}(A^{d - 1}) + \mathcal{N}(A)$ if and only if $d \geq \ind (A)$.
\end{theorem}
\begin{proof}
Assume $d \geq \ind (A)$, or $\mathcal{R}(A^d) \cap \mathcal{N}(A^d) = \lbrace \boldsymbol{0} \rbrace$.
Let $\boldsymbol{b} \in \mathcal{R}(A^d)$ and $\boldsymbol{x}_0 \in \mathcal{R}(A^{d - 1}) + \mathcal{N}(A)$.
Then, $\boldsymbol{r}_0 \in \mathcal{R}(A^d)$ and $\mathcal{K}_k \subseteq \mathcal{R}(A^d)$.
If $k$ is the smallest positive integer such that $\dim \! \mathcal{K}_k < k$, then GMRES determines a solution of $A \boldsymbol{x} = \boldsymbol{b}$ at step $k-1$ (see \cite[Theorem 2.2]{BrownWalker1997}).
Now, assume $\dim \! \mathcal{K}_i = i$, $1 \leq i \leq k$.
Since $\mathcal{K}_{i+1} = \boldsymbol{r}_0 \cup A \mathcal{K}_i$, we have $\dim A \mathcal{K}_i = \dim \mathcal{K}_i = i$ for $i = 1, 2, \dots, k-1$.
Let the columns of $V \in \mathbb{R}^{n \times k}$ form a basis of $\mathcal{K}_k$.
If $\dim A \mathcal{K}_k < \dim \mathcal{K}_k$, then there exists $\boldsymbol{c} \not = \boldsymbol{0}$ such that $A V \boldsymbol{c} = \boldsymbol{0}$.
Since $V \boldsymbol{c} \not = \boldsymbol{0}$ for $\boldsymbol{c} \not = \boldsymbol{0}$, we have $\mathcal{K}_k \cap \mathcal{N}(A) \not = \lbrace \boldsymbol{0} \rbrace$.
From $\mathcal{K}_k \subseteq \mathcal{R}(A^d)$ and $\mathcal{N}(A) \subset \mathcal{N}(A^d)$, we have $\mathcal{R}(A^d) \cap \mathcal{N}(A^d) \not = \lbrace \boldsymbol{0} \rbrace$, which contradicts with $d \geq \ind (A)$.
Hence, $\dim A \mathcal{K}_k = \dim \mathcal{K}_k$ for all $k \in \mathbb{N}$.
Since GMRES does not break down through rank deficiency of the least squares problem $\min_{\boldsymbol{z} \in \mathcal{K}_k} \| \boldsymbol{r}_0 - A \boldsymbol{z} \|$, the sufficiency is shown from \cite[Theorem 2.2]{BrownWalker1997}. 

On the other hand, assume $d < \ind (A)$.
Then, $\mathcal{N}(A^d) \subset \mathcal{N}(A^{d + 1})$.
There exits $\boldsymbol{s} \not = \boldsymbol{0}$ such that $\boldsymbol{s} \not \in \mathcal{N}(A^d)$ and $\boldsymbol{s} \in \mathcal{N}(A^{d + 1})$.
Let $\boldsymbol{t} = A^d \boldsymbol{s}$.
Then, $\boldsymbol{t} \not = \boldsymbol{0}$ and $A \boldsymbol{t} = A^{d + 1} \boldsymbol{s} = \boldsymbol{0}$.
Hence, there exists $\boldsymbol{t} \not = \boldsymbol{0}$ such that $\boldsymbol{t} \in \mathcal{R}(A^d) \cap \mathcal{N}(A)$.
Let $\boldsymbol{b} = \boldsymbol{t} + A \boldsymbol{x}_0$ for $\boldsymbol{x}_0 \in \mathcal{R}(A^{d - 1}) + \mathcal{N}(A)$.
Then, $\boldsymbol{b} \in \mathcal{R}(A^d)$ and $\boldsymbol{r}_0 = \boldsymbol{b} - A \boldsymbol{x}_0 = \boldsymbol{t} \not = \boldsymbol{0}$.
Since $A \boldsymbol{r}_0 = \boldsymbol{0}$, we have $\boldsymbol{r}_1 = \boldsymbol{b} - A \boldsymbol{x}_1 = \boldsymbol{b} - A (\boldsymbol{x}_0 + c \boldsymbol{r}_0) = \boldsymbol{r}_0 - c A \boldsymbol{r}_0 = \boldsymbol{r}_0 \not = \boldsymbol{0}$ for $c \in \mathbb{R}$ and $\dim A \mathcal{K}_1 = 0 < \dim \mathcal{K}_1 = 1$.
Hence, GMRES breaks down at step 1 before determining a solution of $A \boldsymbol{x} = \boldsymbol{b}$.
Therefore, we complete the proof.
\qed
\end{proof}
This theorem agrees with \cite[Theorem 2.2]{MorikuniHayami2015} for $d = 1$.
Although similar results to Theorem \ref{th:GMRESanyind} were given in \cite{WeiWu2000}, no attention was paid there to the uniqueness of the GMRES iterate $\boldsymbol{x}_k$, i.e., the dimensions of $A \mathcal{K}_k$ and $\mathcal{K}_k$.


\section*{acknowledgements}
The author would like to thank Doctor Miroslav Rozlo\v{z}n\'{i}k, Doctor Akira Imakura, and the referees for their valuable comments.


\bibliographystyle{siam}
\bibliography{ref}
\end{document}